\documentclass[a4paper]{article}
\usepackage[utf8]{inputenc}

\usepackage{amsmath, amssymb, amsthm, stackrel,mathtools}
\usepackage{bbm}
\usepackage{titlesec, tikz,pgf}
\usepackage{tikz-cd}
\usepackage[hidelinks]{hyperref}
\usetikzlibrary{arrows}
\usetikzlibrary{knots}
\usetikzlibrary{decorations.markings}

\usepackage[margin=3.9cm]{geometry}
\usepackage{authblk}

\titleformat{\paragraph}[runin]{\normalfont\normalsize\bfseries}{}{\parindent}{}[.]

\newcommand{\rar}{\rightarrow}
\newcommand{\tr}{\textnormal{tr}}

\newcommand{\End}{\textnormal{End}}

\newcommand{\Rep}{\mathbf{Rep}}

\newcommand{\C}{\mathbb{C}}

\newtheorem{thm}{Theorem}

\newtheorem{lem}[thm]{Lemma}
\newtheorem{prop}[thm]{Proposition}

\theoremstyle{definition}

\begin{document}

\title{Exceptional embeddings of $N=2$ minimal models}
\author[1]{Ana Ros Camacho}
\affil[1]{\small Facultat de Ciències Matemàtiques, Universitat de València, Carrer del Doctor Moliner 50, 46100 Burjassot (València), Spain}
\author[2]{Thomas A. Wasserman}
\affil[2]{Aalto University, Espoo, Finland}

\maketitle

\begin{abstract}
Vafa and Warner observed that the Landau-Ginzburg model associated to the potential $E_6$ (resp. $E_8$) is a product of two other models, associated to the potentials $A_2$ and $ A_3$ (resp. $A_2 $ and $ A_4$). We translate this along the Landau-Ginzburg / Conformal Field Theory correspondence to a conjecture about the unitary minimal quotients $M_d$ of the $N=2$ superconformal algebra of central charge $c_d=3-\frac{6}{d}$: there should be a conformal embedding $M_{12}\hookrightarrow M_{3} \otimes M_4$ (resp. $M_{30}\hookrightarrow M_{3} \otimes M_5$) that exhibits the product as Ostrik's $E_6$ (resp. $E_8$) algebra in the $\Rep(su(2)_{10})$ (resp. $\Rep(su(2)_{28})$) factor of the NS-sector of $\Rep(M_{12})$ (resp. $\Rep(M_{30})$). We motivate, formulate, and prove this conjecture. 
\end{abstract}
\setcounter{tocdepth}{3}

\section{Introduction}
The Landau-Ginzburg / Conformal Field Theory (LG/CFT) correspondence appeared in the physics literature in the late 80s (\cite{Fermi1989,Howe1989a,Kastor1989,Vafa1989}). It links LG-models, namely $\left( 2,2 \right)$-supersymmetric sigma models completely characterised by a quasihomogeneous polynomial $W$ called the potential, with $N=2$ CFTs. The so-called non-renormalisation theorem asserts that the potential $W$ is invariant under the renormalisation group flow, yielding a correspondence between the LG-model associated to $W$ and the CFT at the infrared fixed point. At the level of topological defect lines, this predicts equivalences of tensor categories between categories of matrix factorisations of $W$ and  categories of representations of vertex operator algebras (VOAs), respectively. It is an open problem to provide a rigorous mathematical version of the LG/CFT correspondence. This article is part of a program \cite{Brunner2007,Carqueville2016b,Davydov2018,RosCamacho2020,RosCamacho} of work towards this, and provides evidence that the LG/CFT correspondence preserves tensor products of quantum field theories. 
Our result can be phrased purely in terms of VOAs (see Section~\ref{sect: proof} for the notation):
\begin{thm}\label{thm: main theorem}
    There are conformal embeddings between N=2 minimal quotients $M_d$ of central charge $3-\frac{6}{d}$, with $d \geq 2$, 
    \begin{align*}
        M_{12}&\hookrightarrow M_3 \otimes M_4\\
        M_{30}&\hookrightarrow M_3 \otimes M_5.
    \end{align*}
    Furthermore, these embeddings exhibit $M_3\otimes M_4$ and $ M_3 \otimes M_5$ as the algebras $E_6$ and $E_8$ in $\Rep(M_{12})$ and $\Rep(M_{30})$, respectively.
\end{thm}
This paper consists of two sections. Section~\ref{sect: lgcft} explains how Theorem~\ref{thm: main theorem} is predicted by the LG/CFT correspondence. Section~\ref{sect: proof} provides a proof of the main theorem.

\section{The LG/CFT correspondence}\label{sect: lgcft}
To explain how Theorem~\ref{thm: main theorem} is predicted by the LG/CFT correspondence, we use a notion of \emph{orbifolding} and a coincidence of \emph{ADE-patterns} on both sides of the correspondence.

\paragraph{Orbifolding}
On the LG-side {orbifolding} \cite{Carqueville2016a} is formulated in a bicategory which has as its objects potentials and as its morphisms matrix factorisations \cite{Carqueville2016c}. An object $a$ is called an {orbifold} of $b$ if there is a monadic adjunction between them; this implies that $\End(b)$ is equivalent to the category of modules for the monad of the adjunction in $\End(a)$. Potentials related via orbifolding are called {orbifold equivalent}.

On the other hand, full CFTs are determined by a choice of VOA with conformal blocks $\mathcal{H}$ and a consistent choice of correlators in $\mathcal{H}\otimes\overline{\mathcal{H}}$. Any VOA admits a Cardy (or diagonal) choice of correlators. The full Cardy CFT for a VOA $V$ is an orbifold of the full Cardy CFT for $V'$ if there is a conformal embedding $V'\hookrightarrow V$. In the cases we are interested in, such an embedding realises $V$ as an algebra in $\Rep(V')$.

\paragraph{ADE in the LG/CFT correspondence}
In \cite{Vafa1989} the connection between LG-models and singularity theory (for the potential $W$) is explored. The potentials defining simple singularities follow an ADE-classification \cite{Arn,AGV}. We will focus on the $A_{d-1}$, $E_6$ and $E_8$ potentials, these are given by:
\begin{equation*}
\begin{split}
    A_{d-1} &=x^d +y^2,\\
    E_6 &=x^3+y^4, \\
    E_8 &=x^3+y^5,
\end{split}
\end{equation*}
for $d \geq 2$.
The categories of defect lines in the LG-model for a potential $W$ are given by matrix factorisations of $W$ \cite{Brunner2007}. The $E_6$ (resp. $E_8$) potential corresponds via orbifolding to a certain algebra in the category of matrix factorisations for $A_{11}$ (resp. $A_{29}$) \cite{Carqueville2016b}.

The unitary N=2 minimal models also follow an ADE classification \cite{Gannon08,Gray2008}. They are full CFTs with VOA $M_d$ for $d\geq 2$. Such full CFTs are classified by algebra objects in the Neveu-Schwarz (NS) sector of $\Rep(M_d)$. The $A_{d-1}$ model corresponds to the trivial algebra in $\Rep(M_d)$. The $E_6$ and $E_8$ models are orbifolds of $A_{11}$ and $A_{29}$, respectively. The categories $\Rep(M_{12})$ and $\Rep(M_{30})$ have $\Rep(su(2)_{10})$ and $\Rep(su(2)_{28})$) as factors \cite{Davydov2018} and these contain the orbifold algebras $E_6$ and $E_8$ (cf. the classification of algebras in $\Rep(su(2)_{d})$ in \cite{Ostrik2003}).

The LG/CFT correspondence links the LG-models with ADE potentials to the ADE N=2 minimal models. For the $A_d$-type potentials this translates to tensor equivalences between the categories of matrix factorisations for these potentials and the bosonic parts of $\Rep(M_d)$. This was established mathematically in \cite{Davydov2018} for $d$ odd and in \cite{RosCamacho} for arbitrary $d$ and preserves the algebras associated to $E_6$ and $E_8$ \cite{Carqueville2016b}.

\paragraph{The Vafa-Warner prediction}
There is a further observation from \cite{Vafa1989}, saying that the LG-model for $E_6$ (resp. $E_8$) is a product of two other models, $A_2 \otimes A_3$ (resp. $A_2 \otimes A_4$). The justification is simple: taking products of models adds potentials, and Kn\"orrer periodicity \cite[Proposition 1.2]{Carqueville2016b} gives an orbifold equivalence between any potential $W$ and $W+x^2 + y^2$, yielding $A_2+ A_3 \sim E_6$ and $A_2 + A_4 \sim E_8$. Translating this across the LG/CFT correspondence predicts an equivalence between the tensor product of the corresponding minimal models. To prove this equivalence one needs to show that the models $A_2\otimes A_3$ and $A_2 \otimes A_4$ are the $E_6$ and $E_8$ orbifolds of $A_{11}$ and $A_{29}$, respectively. Because the tensor product models are the full Cardy CFTs for the tensor products of the VOAs, this predicts the statement of Theorem~\ref{thm: main theorem}.

\section{Proof of the main theorem}\label{sect: proof}
We first recall some facts about $N=2$ minimal quotients, then we establish the existence of the conformal embeddings from Theorem~\ref{thm: main theorem}. After this we introduce the algebras $E_6$ and $E_8$ and finish the proof of Theorem~\ref{thm: main theorem} by doing a character computation.

\paragraph{$\mathbf{N=2}$ minimal quotients} Let $d \in \mathbb{Z}_{\geq 2}$, we will denote by $V_d$ the universal $N=2$ superVOA of central charge 
$$
c_d=3-\frac{6}{d}.
$$
It is generated by fields $J(z)$, $T(z)$, and $G^\pm(z)$, where $T(z)$ is the energy-momentum tensor, and $J$, $G^+$ and $G^-$ are Virasoro primaries of conformal weight $1$, $\frac{3}{2}$ and $\frac{3}{2}$ respectively \cite{Creutzig2019}. For the given central charges, $V_d$ has a unique nontrivial proper ideal $I_d$ \cite{GK}. The $N=2$ \textit{minimal quotient} $M_d$ is the quotient of $V_d$ by this ideal. 

There is a nice description of the ideal $I_d$ using the vacuum Shapovalov form. This is a pairing on the universal enveloping algebra of a Lie algebra $\mathfrak{g}$ equipped with an anti-involution $(-)^\dagger$ \cite[Definition 3.1]{IoharaKoga}. It is given by $\langle X, Y \rangle:=\pi(X^\dagger Y)$, where $\pi\colon U(\mathfrak{g})\rar \C$ is the projection onto the subspace spanned by $1$. For the mode algebra of $V_d$ one takes $(-)^\dagger$ to be defined by
$$
J_n^\dagger = J_{-n},\quad L_n^\dagger=L_{-n}, \quad (G^\pm_r)^\dagger=G_{-r}^\mp.
$$
where $n \in \mathbb{Z}$ and $r \in \mathbb{Z}+\frac{1}{2}$ in the NS sector. Note that the vacuum Shapovalov pairing can be computed using $(-)^\dagger$ and the commutation relations.

\begin{lem}
    The unique ideal of $V_d$ that yields $M_d$ is the radical of the Shapovalov form.
\end{lem}
\begin{proof}
    The radical of the Shapovalov form is a non-trivial proper ideal \cite[Proposition 3.4]{IoharaKoga}, and hence the unique non-trivial proper ideal of $V_d$.
\end{proof}

\paragraph{The conformal embedding} The conformal embedding descends from the diagonal map on the universal algebras. One checks that the central charges match to show:

\begin{lem}\label{lem: universal embedding}
    The diagonal map $X \mapsto X\otimes 1+ 1\otimes X$ is a conformal embedding $V_{d_1}\hookrightarrow V_{d_2}\otimes V_{d_3}$ exactly when (up to exchanging $d_2$ and $d_3$) we have $d_2=3$, and (i) $d_1=6$ and $d_3=3$, (ii)  $d_1=12$ and $d_3=4$, or (iii)  $d_1=30$ and $d_3=5$.
\end{lem}
We will only be interested in cases (ii) and (iii), case (i) is a simple current extension.

\begin{prop}
      The embeddings from Lemma~\ref{lem: universal embedding} descend to the minimal quotients.
\end{prop}

\begin{proof}
The Shapovalov pairing is completely determined by the commutation relations between the modes and $(-)^\dagger$. Now, the embeddings preserve both $(-)^\dagger$ and the commutation relations, thus also preserve the pairing. In particular, they send the radical of the pairing into the radical of the pairing on the tensor product, and hence descend to an isometric conformal embedding on the quotients.
\end{proof}

This completes the proof of the first part of Theorem~\ref{thm: main theorem}. 

\paragraph{The algebra objects} The VOA $M_d$ is semisimple and has irreducible modules ${}^{[0]}C_{p;r}$ for $r=1,2,\dots, d-1$ and $p=-r,r-+1,\dots, r-1$, which are NS when $p+r$ is odd, with conformal weight $\Delta^{N=2}_{p,r}$ and $J$-highest weight ($J_0$-eigenvalue) $j_{p,r}$ given by \cite[Equation (4.10)]{Creutzig2019}
\begin{align}
    \Delta_{p,r}&=\frac{r^2-p^2-1}{4d}+\frac{1+\left( -1 \right)^{r+p}}{16}\label{eq: conformal weights}\\
    j_{p,r}&= \frac{p}{d}+ \frac{1+(-1)^{p+r}}{4}.\label{eq: jweight}
\end{align}
Writing $C_{r}:={}^{[0]}C_{0;r}$ and using \cite{Carqueville2016b,Davydov2018}, the algebras in Theorem~\ref{thm: main theorem} are
\begin{align}\label{eq: algs}
\begin{split}
 E_6&= C_{1} \oplus C_{7} \in \Rep (M_{12})\\
    E_8&= C_{1} \oplus C_{11}\oplus C_{19}\oplus C_{29} \in \Rep (M_{30}).
\end{split}
\end{align}
The characters of the $C_r$ are \cite{Dob,Mat} (see also \cite[Equation (5.17)]{Creutzig2019}):
\begin{equation}\label{eq: characters}
    \tr_{C_r}(q^{L_0})= \frac{q^{\Delta^{N=2}_{0,r}+\frac{1}{8}}}{\eta \left( q \right)^3} \sum\limits_{j \in \mathbb{Z}} \left( \frac{\vartheta_3 \left( 1;q \right)}{1+ q^{(2dj+r)/2}} -\frac{\vartheta_3 \left( 1;q \right)}{1+ q^{-(2dj+r)/2}} \right)q^{j(dj+r)}
\end{equation}
where $\eta \left( q \right)=q^{\frac{1}{24}}\prod_{i=1}^{\infty}(1-q^{i})$ is the Dedekind eta function and $\vartheta_3 \left( 1;q \right)=\prod_{i =1}^{\infty}(1+q^{i-\frac{1}{2}})^2(1-q^i)$ is a Jacobi theta function.
\begin{prop}
    We have 
    \begin{align*}
        M_3 \otimes M_4 &\cong C_{1} \oplus C_{7}\quad \quad \quad \quad\quad \quad\mbox{ and }\\
        M_3 \otimes M_5 & \cong C_{1} \oplus C_{11}\oplus C_{19}\oplus C_{29}
    \end{align*}
    as $M_{12}$ and $M_{30}$ modules, respectively.
\end{prop}

\begin{proof}
    The conformal weights of vectors in $M_d$ are integers and half-odd integers, so the same is true for $M_{3}\otimes M_{4}$ and $M_{3}\otimes M_{5}$. The J-weights are integers, acting by $L_n$ and $J_n$ modes leaves the J-weight invariant, acting by $G^\pm_r$ changes the J-weight by $\pm 1$. From this we also see that homogeneous vectors with odd J-weight have half-odd conformal weight, even J-weight implies integral conformal weight.

    Consider the highest weight vectors in $M_{3}\otimes M_{4}$ and $M_{3}\otimes M_{5}$. As these generate modules ${}^{[0]}C_{p;r}$ in the NS sector we have $p+r$ odd. For an integral $J$-weight Equation \eqref{eq: jweight} we then require $p=0$, so both $r$ and the $J$-weight are even. Thus we need an integral conformal weight Equation \eqref{eq: conformal weights}, meaning that $r^2-1$ is divisible by 48 or 120, respectively. The possibilities are $r=1, 7$ (with conformal weights 0 and 1, respectively) for $M_{12}$ and $r=1,11,19,29$ (with conformal weights 0, 1, 3, and 7, respectively) for $M_{30}$, respectively.

    To show that the possible modules all occur once we compare the dimensions of the $L_0$-eigenspaces on both sides of Equation \eqref{eq: algs}. On the tensor product $L_0\in M_{12}$ acts as the sum of the $L_0$'s in each factor. So we have
    $$
    \tr (q^{L_0}|_{M_3\otimes M_4}) = \tr(q^{L_0}|_{M_3})\tr(q^{L_0}|_{M_4}),
    $$
    and a similar formula for $L_0\in M_{30}$ acting on $M_3\otimes M_5$. The characters are given in Equation \eqref{eq: characters}. For the $M_{12}$-case it suffices to expand to degree 1 -- this shows that the conformal weight 0 eigenspace is one dimensional and the weight 1 eigenspace is two dimensional. The module $C_0$ has a one-dimensional subspace of conformal weight 0, while $C_0$ and $C_7$ both have one vector of conformal weight 1, so both appear once.

    For $M_{30}$, we expand up to degree seven in $q$. This gives for the $L_0$-eigenspaces
        \begin{center}
        \begin{tabular}{c|c|c|c|c}
             Conformal weight& 0 & 1 & 3 & 7 \\ \hline
             Dimension in $M_3\otimes M_5$&1 &2 & 18 & 496\\ \hline
             Dimension in $C_1$&1 &1 & 6 & 107\\ \hline
             Dimension in $C_{11}$&0 &1 & 11 & 319\\ \hline
             Dimension in $C_{19}$&0 &0 & 1 & 69\\ \hline
             Dimension in $C_{29}$&0 &0 & 0 & 1\\ 
        \end{tabular}       .
        \end{center}
    The lower four rows add up to the first row, so each module occurs once. 
\end{proof}

This finishes the proof of Theorem~\ref{thm: main theorem}.

\subsubsection*{Acknowledgements}
The authors are very thankful to Simon Wood for helpful discussions and advice, and grateful to Shenji Koshida, Gabriel Navarro and David Ridout for comments on a draft of this paper. We acknowledge support from the London Mathematical Society through ARC's Emmy Noether Fellowship EN-2324-02. TW acknowledges support from the European Commission through a Marie Sklodowska-Curie Individual Fellowship with reference number 101203556 -- DisQ FA, Richard Wade's Royal Society of Great Britain University Research Fellowship, and the University of Oxford. ARC thanks Cardiff University and their University Research Leave 2025/26 Scheme for support.

\bibliography{library}
\bibliographystyle{alpha}

\end{document}